\documentclass[12pt]{amsart}

\usepackage{graphics}
\usepackage{amsmath}
\usepackage{amsfonts}
\usepackage{amssymb}
\usepackage{amscd}

\newcommand{\C}{{\mathbb C}}
\newcommand{\R}{{\mathbb R}}
\renewcommand{\O}{{\mathcal O}}
\newcommand{\Om}{{\Omega}}

\newtheorem{theorem}{\bf Theorem}

\newtheorem{proposition}[theorem]{\bf Proposition}
\newtheorem{corollary}[theorem]{\bf Corollary}

\title{Schlicht envelopes of holomorphy \\ and foliations by lines}

\author{Finnur L\'arusson}
\address{School of Mathematical Sciences, University of Adelaide, Adelaide SA 5005, Australia.} 
\email{finnur.larusson@adelaide.edu.au}

\author{Rasul Shafikov}
\address{Department of Mathematics, University of Western Ontario, London, Ontario N6A~5B7, Canada.} 
\email{shafikov@uwo.ca}


\date{10 February 2008.  Minor changes 7 August 2008}

\begin{document}

\begin{abstract}  Given a domain $Y$ in a complex manifold $X$, it is a difficult problem with no general solution to determine whether $Y$ has a schlicht envelope of holomorphy in $X$, and if it does, to describe the envelope.  The purpose of this paper is to tackle the problem with the help of a smooth 1-dimensional foliation $\mathcal F$ of $X$ with no compact leaves.  We call a domain $Y$ in $X$ an {\it interval domain} with respect to $\mathcal F$ if $Y$ intersects every leaf of $\mathcal F$ in a nonempty connected set.  We show that if $X$ is Stein and if $\mathcal F$ satisfies a new property called {\it quasiholomorphicity}, then every interval domain in $X$ has a schlicht envelope of holomorphy, which is also an interval domain.  This result is a generalization and a global version of a well-known lemma from the mid-1980s.  We illustrate the notion of quasiholomorphicity with sufficient conditions, examples, and counterexamples, and present some applications, in particular to a little-studied boundary regularity property of domains called local schlichtness.
\end{abstract}

\maketitle

\tableofcontents

\section{Introduction}

\noindent
The roots of the subject of several complex variables go back to the discovery, a century ago, that in contrast to the 1-dimensional case, there are domains $\Om$ in higher dimensions such that all holomorphic functions on $\Om$ extend holomorphically to a larger domain.  This leads directly to the notion of the envelope of holomorphy.  The {\it envelope of holomorphy} of a complex manifold $Y$ is a Stein manifold $Z$ with a holomorphic embedding of $Y$ onto a domain in $Z$ that induces a bijection between the sets of holomorphic functions on $Y$ and on $Z$.  If it exists, the envelope is uniquely determined up to isomorphism, and the assignment of its envelope to a manifold that has one is functorial.  If $Y$ is a domain in a complex manifold $X$ and there is a Stein domain $Z$ in $X$ containing $Y$ to which all holomorphic functions on $Y$ extend holomorphically, then $Z$ is the envelope of holomorphy of $Y$, and we say that $Y$ has a {\it schlicht} envelope of holomorphy in $X$.

Given a domain $Y$ in a complex manifold $X$, it is a difficult problem with no general solution to determine whether $Y$ has a schlicht envelope of holomorphy in $X$, and if it does, to describe the envelope.  The purpose of this paper is to tackle the problem with the help of a smooth 1-dimensional foliation $\mathcal F$ of $X$.  

The rough idea is to continue the holomorphic functions on $Y$ along the leaves of $\mathcal F$ (rather than along arbitrary curves in $X$) as far away from $Y$ as possible, and obtain the envelope of holomorphy of $Y$ as the union of the maximal leaf segments through $Y$ along which all holomorphic functions on $Y$ can be continued.  It is natural to require the leaves of $\mathcal F$ to be noncompact, that is, exclude embedded circles as leaves, and to insist that $Y$ intersect each leaf, which is then an injectively immersed real line, in a nonempty connected set.  Namely, as we continue holomorphic functions beyond $Y$ along leaves of $\mathcal F$, we do not want to come back to where we have already been.  And if $Y$ does not intersect every leaf, we might as well replace $X$ by the saturation of $Y$ with respect to $\mathcal F$.

With these requirements, the rough idea works, as long as $\mathcal F$ satisfies a new property that we call {\it quasiholomorphicity}.  Continuation of holomorphic functions on $Y$ along leaves of $\mathcal F$ produces a locally Stein domain $Z$ in $X$ containing $Y$ to which all holomorphic functions on $Y$ extend holomorphically.  If $X$ is Stein, so is $Z$, so $Y$ has a schlicht envelope of holomorphy in $X$, about which we have the additional information that it intersects every leaf of $\mathcal F$ in a nonempty connected set.

This result, stated more precisely below as Theorem \ref{Trepreau1}, can be viewed as a generalization and a global version of a well-known lemma from the mid-1980s (Theorem \ref{Trepreau0} below) that we shall call the schlichtness lemma.  This lemma corresponds to our result with $X$ being a box in $\C^n$ foliated by straight line segments parallel to one of its edges.

Following the proof of our global schlichtness lemma, we illustrate the notion of quasiholomorphicity with sufficient conditions, examples, and counterexamples, and present some applications, in particular to a little-studied boundary regularity property of domains called local schlichtness.

\medskip

Let us summarize the contents of the paper in more detail.  We need a few new definitions.  Let $\mathcal F$ be a smooth 1-dimensional foliation of a complex manifold $X$ with no compact leaves.  Then the leaves of $\mathcal F$ are injectively immersed lines (not necessarily embedded).  A domain $\Om$ in $X$ is called an {\it interval domain} with respect to $\mathcal F$ if $\Om$ has a nonempty connected intersection with each leaf of $\mathcal F$.

Let $Y$ be a domain in $X$ with a connected, but possibly empty, intersection with each leaf of $\mathcal F$.  Then $Y$ is an interval domain in its saturation $Y^\mathcal F$ (the union of all leaves that intersect $Y$; this is a domain in $X$).  It is easily shown that there is a largest interval domain $Z$ in $Y^\mathcal F$ containing $Y$ to which all holomorphic functions on $Y$ extend holomorphically (see the first paragraph of the proof of Theorem \ref{Trepreau1} below).  We call $\mathcal F$ a {\it good foliation} if $Z$ is locally Stein in $Y^\mathcal F$ for all $Y$ (meaning that every boundary point of $Z$ in $Y^\mathcal F$ has an open neighbourhood $U$ in $Y^\mathcal F$ such that $Z\cap U$ is Stein).

Next, we say that $\mathcal F$ is {\it quasiholomorphic} if it satisfies the following two properties:
\begin{enumerate}
\item if $p,q\in X$ lie in the same leaf, then there is an open neighbourhood $U$ of $p$ in $X$ and a biholomorphism $h$ from $U$ onto an open neighbourhood of $q$, taking $p$ to $q$, such that $h(x)$ lies in the leaf of $x$ for all $x\in U$; and
\item if $a\in U$ and $b$ is a point in the leaf of $a$ between $a$ and $h(a)$, then there is a biholomorphism $k$ from $U$ onto an open neighbourhood of $b$, taking $a$ to $b$, such that $k(x)$ lies in the leaf of $x$ between $x$ and $h(x)$ for every $x\in U$.
\end{enumerate}
Note that quasiholomorphicity is a semilocal property in the sense that $\mathcal F$ is quasiholomorphic if and only if $\mathcal F$ is quasiholomorphic on a saturated open neighbourhood of each of its leaves.

Finally, we say that $\mathcal F$ has a {\it holomorphic atlas} if $X$ is covered by holomorphic charts that take $\mathcal F$ to the foliation of $\C^n$ ($n=\dim X$) by the straight lines along which $\mathrm{Im}\,z_1$ and $z_2,\dots,z_n$ are constant.  By the rectification theorem for holomorphic vector fields, $\mathcal F$ has a holomorphic atlas if and only if $\mathcal F$ is induced by a nowhere-vanishing holomorphic vector field on a neighbourhood of each point of $X$.  This is of course a local property.  We do not know whether being good is a local or semilocal property.

\medskip\noindent
{\bf Main Theorem.}  {\it Let $\mathcal F$ be a smooth 1-dimensional foliation of a complex manifold $X$ with no compact leaves.  If $\mathcal F$ is quasiholomorphic, then $\mathcal F$ is good.  If $\mathcal F$ has a holomorphic atlas, then $\mathcal F$ is quasiholomorphic.}
\medskip

We do not know whether or to what extent the converses of these implications are true (except in complex dimension 1, where every smooth 1-dimensional foliation with no compact leaves is good for trivial reasons, but need not be quasiholomorphic, as shown by examples given below).  As a corollary, we obtain a generalization of the schlichtness lemma.

\medskip\noindent
{\bf Global Schlichtness Lemma.}  {\it Let $\mathcal F$ be a quasiholomorphic foliation on a complex manifold $X$, and let $Y$ be an interval domain in $X$ with respect to $\mathcal F$.  

There is a largest interval domain $Z$ in $X$ with respect to $\mathcal F$ containing $Y$ to which all holomorphic functions on $Y$ extend holomorphically.  Moreover, $Z$ is locally Stein in $X$.  

Hence, if $X$ is Stein, so is $Z$, so $Z$ is the envelope of holomorphy of $Y$.  In particular, if $X$ is Stein, then $Y$ has a schlicht envelope of holomorphy in $X$.}
\medskip

In general, if $X$ is not Stein, $Z$ being locally Stein does not imply that $Z$ is the envelope of holomorphy of $Y$.  However, there are positive results on the Levi problem for certain non-Stein manifolds, and for such manifolds further conclusions can be drawn from $Z$ being locally Stein.

We call a domain $\Om$ in a complex manifold {\it locally schlicht} at a boundary point $p$ if $p$ has a basis of connected Stein neighbourhoods $U$ such that $\Om\cap U$ has a schlicht envelope of holomorphy in $U$.  As shown below, it follows from the schlichtness lemma that a domain is locally schlicht at a smooth boundary point.  The original motivation for our work was to extend this result to singular boundary points.  Using the global schlichtness lemma, we are able to prove local schlichtness at well-behaved isolated boundary singularities.

\medskip\noindent
{\bf Local Schlichtness Theorem.}  {\it Let $\rho$ be a smooth real-valued function on a neighbourhood of the origin $0$ in $\C^n$.  Suppose $\rho$ has a nondegenerate hermitian critical point at $0$, which is not a minimum.  Then the open set $\{\rho<\rho(0)\}$ is locally schlicht at $0$.}
\medskip

The hermitian condition means that the linear part of the Taylor expansion of the gradient of $\rho$ at $0$ is complex-linear.

\section{The schlichtness lemma and local schlichtness at a boundary point}

\noindent
The work presented in this paper starts with a lemma from the mid-1980s that we shall call the schlichtness lemma.  The following statement is a slight modification of the result proved by Jean-Marie Tr\'epreau as Lemma 1.2 in his paper \cite{Trepreau}.  Tr\'epreau says about this result:  \lq\lq Le lemme suivant est, sur le fond, bien connu des sp\'ecialistes.\rq\rq\  The schlichtness lemma also appeared in the unpublished Ph.D.\ thesis of Berit Stens\o nes, completed in 1985, see Theorem 2 in \cite{Rea} and the proof of Theorem 2.7 in \cite{Stensones}.  See also Lemma 1 in \cite{Chirka}.  A related result appeared much earlier in \cite{Vladimirov}, Section 21.5.  The norms used below are the maxima over the absolute values of the relevant real coordinates.

\begin{theorem}[The Schlichtness Lemma]  Let $Y=\{(z,x)\in\C^{n-1}\times\R:|(z,x)|<r\}$ and $X=Y\times (-r,r)$, where $r>0$.  Let $u:Y\to\R$ be a lower semicontinuous function with $|u|<r/3$ and let $\Om=\{(z,x+iy)\in X:y<u(z,x)\}$.  Then there is a lower semicontinuous function $v:Y\to\R$ with $v\geq u$ such that the domain $\tilde\Om=\{(z,x+iy)\in X:y<v(z,x)\}$ is the envelope of holomorphy of $\Om$, meaning that $\tilde\Om$ is Stein and every holomorphic function on $\Om$ extends to a holomorphic function on $\tilde\Om$.
\label{Trepreau0}
\end{theorem}

In \cite{Trepreau}, $u$ is taken to be $C^2$, but this is not necessary.  We will not recall the proof of the schlichtness lemma.  It is subsumed by Theorem \ref{Trepreau1} below.  The hypothesis $|u|<r/3$ is used in \cite{Trepreau}, but Theorem \ref{Trepreau1} shows that it is not necessary either.

Now let $\Om$ be an open subset of a complex manifold $X$.  We say (by a slight abuse of terminology) that $\Om$ is {\it locally schlicht} at a boundary point $p$ of $\Om$ in $X$ if $p$ has a basis of open neighbourhoods $U$ such that:
\begin{enumerate}
\item $U$ is connected and Stein, that is, $U$ is a domain of holomorphy,
\item there is a Stein open set $V$ with $\Om\cap U\subset V\subset U$, and
\item every holomorphic function on $\Om\cap U$ extends to a holomorphic function on $V$.
\end{enumerate}
Then $V$ is the envelope of holomorphy of $\Om\cap U$.  Note that $\Om\cap U$ and $V$ may be disconnected.  Note also that if $\Om$ is Stein, then $\Om$ is locally schlicht at each of its boundary points.  More generally, if $\Om$ is pseudoconvex at a boundary point $p$, meaning that there is an open neighbourhood $W$ of $p$ such that $\Om\cap W$ is Stein, then $\Om$ is locally schlicht at $p$.

It follows from the schlichtness lemma that if $\Om$ is the subgraph of a lower semicontinuous function in some holomorphic coordinates at $p$, and this function is continuous at $p$, then $\Om$ is locally schlicht at $p$.  In particular, if the boundary of $\Om$ is smooth at $p$, then $\Om$ is locally schlicht at $p$.

\begin{corollary}  Let $u:\C^{n-1}\times\R\to\R$ be a lower semicontinuous function.  If $u$ is continuous at $(a,s)\in \C^{n-1}\times\R$, then the domain
$$\Om=\{(z,x+iy)\in\C^n:y<u(z,x)\}$$
is locally schlicht at the boundary point $(a,s+iu(a,s))$.
\label{subgraph}
\end{corollary}

\begin{proof}  We may assume that $(a,s)$ is the origin and $u(a,s)=0$.  By assumption, for every $\epsilon>0$, there is $\delta_\epsilon>0$, say $\delta_\epsilon<\epsilon$, such that if $|w|<\delta_\epsilon$, then $|u(w)|<\epsilon$.

Fix $0<\epsilon<1$.  Let $W=\{w\in\C^{n-1}\times\R:|w|<\delta_\epsilon\}$ and 
$$U=\{(z,x+iy)\in\C^n:|(z,x)|<\delta_\epsilon, |y|<3\epsilon\}\cong W\times (-3\epsilon,3\epsilon).$$
We have $|u|<\epsilon$ on $W$.  The box $U$ is a domain of holomorphy.  As $\epsilon$ ranges through the interval $(0,1)$, these boxes form a neighbourhood basis for the origin.

Theorem \ref{Trepreau0} (or Theorem \ref{Trepreau1}, if you worry about $\delta_\epsilon$ and $3\epsilon$ not being equal) now implies that there is a domain of holomorphy $\tilde\Om$ with $\Om\cap U\subset \tilde\Om\subset U$ to which every holomorphic function on $\Om$ extends.
\end{proof}

The continuity assumption in the statement of Corollary \ref{subgraph} cannot be omitted.  Whether it could somehow be relaxed is an open question.  Namely, let $\Om_0$ be a domain in $\C^2$, let $u(z,x)$ equal $2$ if $z\in\Om_0$ and $0$ if $z\not\in\Om_0$, and let $\Om=\{(z,x+iy)\in\C^3:y<u(z,x)\}$.  If $p\in\partial\Om_0$, then $(p,i)\in\partial\Om$, and $\Om=\Om_0\times\C$ in a neighbourhood of $(p,i)$.  Thus, if $\Om_0$ is not locally schlicht at $p$, and as pointed out below, such domains exist, then $\Om$ is not locally schlicht at $(p,i)$.

\smallskip

We conclude this section with two examples of domains $\Om$ that fail to be locally schlicht at a boundary point $p$.  In the second example, $\Om$ is locally connected at $p$, in the first it is not.  Let $\Om$ be the union of the complement of the closed unit ball in $\C^n$ or its intersection with some open neighbourhood of $p=(1,0,\dots,0)$, the open ball of radius $\frac 1 2$ centred at $(\frac 1 2,0,\dots,0)$, and, to make $\Om$ connected, a fattened path, say, joining these two sets away from $p$.   Then $\Om$ is clearly not locally schlicht at $p$.

The second example is a modification of a classical example due to H.\ Cartan (see \cite{Cartan}; for more details see \cite{Narasimhan}, pp.\ 97--98).  Define a domain $\Om_0$ in $\C^2$ as $\Om_1\cup\Om_2$, where
$$\Om_1=\{(x+iy,w)\in\C^2 : -4<x<0,\ y>1,\ y|w| < e^x\},$$
$$\Om_2=\{(x+iy,w)\in\C^2: 0\leq x<4,\ y>1, \ e^{-1/x}<y|w|<1\}.$$
The envelope of holomorphy of $\Om_0$ is the domain $\Om_1\cup\widetilde \Om_2$, where 
$$\widetilde\Om_2=\{(x+iy,w)\in\C^2: 0\leq x<4,\ y>1,\ y|w|<1\}.$$
Let $\Om=\phi(\Om_0)\subset\C^2$, where $\phi(z,w)=(e^{iz},w)$. Note that $\phi:\Om_0\to\Om$ is a biholomorphism, the origin $0$  is a boundary point of the domain $\Om$, and $\Om$ is locally connected at $0$.  We claim that $\Om$ is not locally schlicht at $0$.  Namely, let $U$ be an open neighbourhood of $0$ and let $P$ be a polydisc centred at $0$ of polyradius $\big(\epsilon,\dfrac 1 {-\log\epsilon}\big)$ with $\epsilon>0$ so small that $P\subset U$.  Then
$$\{(z,w)\in \Om_0:y>-\log\epsilon\} \subset \phi^{-1}(U\cap \Omega),$$
so any holomorphic function on $U\cap \Om$ admits an analytic continuation along any path in
$$\phi\big(\{(z,w)\in \Om_1\cup\tilde\Om_2: y>-\ln\epsilon\}\big).$$
However, the holomorphic function $z\circ\phi^{-1}$ on $\Om$ (the first component of the inverse of $\phi$) does not extend to a single-valued holomorphic function on any Stein open set $V$ with $\Om\cap U\subset V\subset U$.  Therefore the envelope of holomorphy of $U\cap \Omega$ is not schlicht.

\section{Quasiholomorphic foliations and the global schlichtness lemma}

\noindent
We begin this section by defining two new notions that we need for our generalization of the schlichtness lemma.  Let $\mathcal F$ be a smooth 1-dimensional foliation of a complex manifold $X$.  The compact leaves of $\mathcal F$ are embedded circles and the noncompact leaves are injectively immersed lines (not necessarily embedded).  We say that $\mathcal F$ is {\it quasiholomorphic} if 
\begin{enumerate}
\item $\mathcal F$ has no compact leaves;
\item if $p,q\in X$ lie in the same leaf, then there is an open neighbourhood $U$ of $p$ in $X$ and a biholomorphism $h$ from $U$ onto an open neighbourhood of $q$, taking $p$ to $q$, such that $h(x)$ lies in the leaf of $x$ for all $x\in U$; and
\item if $a\in U$ and $b$ is a point in the leaf of $a$ between $a$ and $h(a)$, then there is a biholomorphism $k$ from $U$ onto an open neighbourhood of $b$, taking $a$ to $b$, such that $k(x)$ lies in the leaf of $x$ between $x$ and $h(x)$ for every $x\in U$.
\end{enumerate}

Under an additional regularity condition on the foliation, the biholomorphisms $h$ and $k$ above are uniquely determined as long as $U$ is connected.  Namely, let $X$ be an $n$-dimensional complex manifold with a smooth 1-dimensional foliation $\mathcal F$ with no compact leaves.  Then $\mathcal F$ has no holonomy since its leaves are simply connected, so its graph, as a set, is simply the set $G$ of pairs $(x,y)$ in $X\times X$ such that $x$ and $y$ lie in the same leaf.  For the definition of the graph of foliation and its basic properties, see \cite{Winkelnkemper}.  The graph has its own topology, possibly finer than the subspace topology induced from $X\times X$, and the structure of a smooth manifold of dimension $2n+1$.  The inclusion $\iota:G\hookrightarrow X\times X$ is a smooth injective immersion.  Suppose that $\iota$ is also proper, that is, an embedding; in other words, that $G$ is a smooth submanifold of $X\times X$.  If $U$ is an open subset of $X$, and $h:U\to X$ is a holomorphic map taking each leaf into itself, then the graph of $h$ is an $n$-dimensional complex submanifold of $U\times X$ and is a subset of $G$.  The tangent space to $G$ at each point has real dimension $2n+1$ and thus contains at most one complex subspace of complex dimension $n$.  If $U$ is connected, $h$ is therefore determined by its value at any one point.

Let $X$ be a smooth manifold, $A$ be a closed subset of $X$, and $\mathcal F$ be a smooth 1-dimensional foliation of $X\setminus A$ with no compact leaves.  By a domain in $X$ we mean, as usual, a nonempty, connected, open subset of $X$.  We call a domain $\Om$ in $X$ an {\it interval domain with respect to} $\mathcal F$, or simply an {\it interval domain} if $\mathcal F$ is understood, if $\Om$ contains $A$ and has a nonempty connected intersection with each leaf of $\mathcal F$.

Conditions (2) and (3) in the definition of quasiholomorphicity are designed to fit naturally into the proof of the following generalization of the schlichtness lemma.

\begin{theorem}[Global Schlichtness Lemma]  Let $\mathcal F$ be a quasiholomorphic foliation on an open subset of a complex manifold $X$, and let $Y$ be an interval domain in $X$ with respect to $\mathcal F$.  

There is a largest interval domain $Z$ in $X$ with respect to $\mathcal F$ containing $Y$ to which all holomorphic functions on $Y$ extend holomorphically.  Moreover, $Z$ is locally Stein in $X$.  

Hence, if $X$ is Stein, so is $Z$, so $Z$ is the envelope of holomorphy of $Y$.  In particular, if $X$ is Stein, then $Y$ has a schlicht envelope of holomorphy in $X$.
\label{Trepreau1}
\end{theorem}

To say that $Z$ is locally Stein in $X$ means that $X$ has a cover by open subsets $U$ such that $Z\cap U$ is Stein.  Equivalently, every boundary point of $Z$ in $X$ has an open neighbourhood $V$ in $X$ such that $Z\cap V$ is Stein.

This result subsumes the schlichtness lemma.  Namely, Theorem \ref{Trepreau0} follows from Theorem \ref{Trepreau1} if we take $X$ to be a product of squares (so $X$ is biholomorphic to a polydisc) and $\mathcal F$ to be the foliation by line segments parallel to a side of one of the squares.

\begin{proof}  Let $\mathcal E$ be the set of all interval domains containing $Y$ to which all holomorphic functions on $Y$ extend.  Note that if $U,V\in\mathcal E$, then $U\cap V$ is connected.  Let $Z=\bigcup \mathcal E$.  Then $Z$ is an interval domain.  Let $f\in\O(Y)$.  For each $U\in\mathcal E$, there is $f_U\in\O(U)$ with $f_U|Y=f$.  For $U, V\in\mathcal E$, $f_U$ and $f_V$ agree on $Y$, so they agree on the connected set $U\cap V$.  Hence the functions $f_U$, $U\in\mathcal E$, define a holomorphic extension of $f$ to $Z$.  This shows that $Z$ is the largest element of $\mathcal E$.

It remains to prove that $Z$ is locally Stein.  Assume to the contrary that $Z$ has a boundary point $p$ such that $Z\cap U$ is not Stein for any open neighbourhood $U$ of $p$.  In other words, there is a basis of open Stein neighbourhoods $U$ of $p$ with a domain $V\subset U$, $V\not\subset Z$, and a domain $W$ in $Z\cap V\neq\varnothing$ such that every holomorphic function on $Z\cap U$ agrees on $W$ with a holomorphic function on $V$.  

Find $q\in Z$ in the leaf of $p$.  Let $U$ be an open neighbourhood of $p$ with a biholomorphism $h$ onto an open neighbourhood of $q$ as in the definition of quasiholomorphicity.  By shrinking $U$ if necessary, we may assume that $h(U)\subset Z$.  By assumption, there is a domain $V\subset U$, $V\not\subset Z$, and a domain $W$ in $Z\cap V\neq\varnothing$ such that every holomorphic function on $Z\cap U$ agrees on $W$ with a holomorphic function on $V$.  Since $V$ is connected, $W$ has a boundary point $a$ in $V$.  We may assume that $W$ is a connected component of $Z\cap V$.  Then $a$ is a boundary point of $Z$.  

Let $L$ be the leaf of $a$.  Since $a\notin Z$ and $h(a)\in Z$, the interval $Z\cap L$ in $L$ has a boundary point $b$ between $a$ and $h(a)$.  Then $b$ is a boundary point of $Z$.  Note that $Z$ is locally connected at $b$.  Let $k$ be a biholomorphism from $U$ onto an open neighbourhood of $b$ as in the definition of quasiholomorphicity.  If $x\in Z\cap U$, then $k(x)$ lies between $x$ and $h(x)\in Z$, so since $Z$ is an interval domain, $k(x)\in Z$.  Thus, $k(Z\cap U)\subset Z\cap k(U)$.

Every holomorphic function on $Z\cap U$ agrees on $W$ with a holomorphic function on $V$.  Hence every holomorphic function on $Z\cap k(U)\supset k(Z\cap U)$ agrees on $k(W)$ with a holomorphic function on $k(V)$.  Find an open neighbourhood $T\subset k(V)$ of $b$ such that $Z\cap T$ is connected.  Then every holomorphic function $f$ on $Z$ agrees on the nonempty open set $k(W)\cap T \subset Z\cap T$ with a holomorphic function on $T$, so $f$ extends holomorphically to $Z\cup T$.  Since $Z\cup T$ contains an interval domain strictly larger than $Z$ itself, this contradicts the definition of $Z$.
\end{proof}

Theorem \ref{Trepreau1} may be extended in various ways.  There are positive results on the Levi problem for manifolds that are not Stein.  For instance, if $\Om$ is a locally Stein domain in a complex projective space or, more generally, in a Grassmannian $\mathbb G$, then $\Om$ is Stein or $\Om=\mathbb G$ \cite{Ueda}.  Let $Y$ be a domain in $\mathbb G$ and suppose $Y$ is an interval domain with respect to a quasiholomorphic foliation on an open subset of $\mathbb G$.  Then, by Theorem \ref{Trepreau1}, either $Y$ has a Stein domain in $\mathbb G$ as its envelope of holomorphy, or all holomorphic functions on $Y$ are constant.

Also, Theorem \ref{Trepreau1} generalizes, with essentially the same proof, from the trivial line bundle $\O$ to any other natural holomorphic vector bundle $E$ on $X$, such as a holomorphic tensor bundle or a jet bundle.  We simply make the neighbourhood $U$ in the proof small enough that $E$ is trivial on $U$.  Then $E$ is also trivial on $k(U)$ by naturality.  We conclude that there is a largest interval domain $Z$ in $X$ containing $Y$ to which all holomorphic sections of $E$ over $Y$ extend holomorphically, and that $Z$ is locally Stein.

\smallskip

The principal example of a quasiholomorphic foliation is the foliation induced by a holomorphic vector field with no compact orbits.  The biholomorphisms $h$ and $k$ in the definition of quasiholomorphicity are then given by time-maps of the flow of the vector field.  Our next result gives a more general sufficient condition for quasiholomorphicity.  We do not know whether this condition is also necessary.

Let $\mathcal F$ be a smooth 1-dimensional foliation of an $n$-dimensional complex manifold $X$.  A chart for $\mathcal F$ is a diffeomorphism from an open subset $U$ of $X$ onto an open subset $V$ of $\C^n$ that takes $\mathcal F|U$ to the foliation of $V$ by the straight lines along which $\mathrm{Im}\,z_1$ and $z_2,\dots,z_n$ are constant.  An atlas for $\mathcal F$ is a set of charts whose domains cover $X$.  An atlas is called holomorphic if all the charts in it are holomorphic.  By the rectification theorem for holomorphic vector fields, $\mathcal F$ has a holomorphic atlas if and only if each point of $X$ has a neighbourhood on which $\mathcal F$ is induced by a nowhere-vanishing holomorphic vector field.  (If $\mathcal F$ has a holomorphic atlas, then $\mathcal F$ is a refinement of a holomorphic foliation of $X$ of complex dimension $1$ with another refinement orthogonal to $\mathcal F$.  We have not found a role for this additional structure here.)

\begin{theorem}  Let $\mathcal F$ be a smooth 1-dimensional foliation of a complex manifold $X$ with no compact leaves.  If $\mathcal F$ has a holomorphic atlas, then $\mathcal F$ is quasiholomorphic.
\label{holomorphicatlas}
\end{theorem}

\begin{proof}  We need to verify conditions (2) and (3) in the definition of quasiholomorphicity.  Suppose $p$ and $q$ lie in the same leaf $L$ of $\mathcal F$.  Cover the compact segment between $p$ and $q$ in $L$ by open sets $V_1,\dots,V_m$ such that the following hold.
\begin{enumerate}
\item[(a)]  $V_1$ is a neighbourhood of $p$, and $V_m$ is a neighbourhood of $q$.
\item[(b)]  $\mathcal F|V_j$ is induced by a nowhere-vanishing holomorphic vector field $F_j$ for $j=1,\dots,m$.
\item[(c)]  The intersection $V_j\cap V_{j+1}$ is connected for $j=1,\dots,m-1$.  To see that this is possible, take $V_1,\dots,V_m$ to be \lq\lq slices\rq\rq\ of a tubular neighbourhood of a relatively compact open segment $I$ in $L$ containing $p$ and $q$, viewing $I$ as a locally closed submanifold of $X$.  Thus we may assume that $F_j$ and $F_{j+1}$ are positive multiples of each other at each point of $V_j\cap V_{j+1}$.
\item[(d)]  There is an open neighbourhood $U$ of $p$ in $V_1$ such that the flow $\phi_t^1$ of $F_1$ is defined on $U$ for $t\in[0,t_1]$, and $U_1=\phi_{t_1}^1(U)\subset V_1\cap V_2$.
\item[(e)]  For $j=2,\dots,m-1$, the flow $\phi_t^j$ of $F_j$ is defined on $U_{j-1}$ for $t\in[0,t_j]$, and $U_j=\phi_{t_j}^j(U_{j-1})\subset V_j\cap V_{j+1}$.
\item[(f)]  The flow $\phi_t^m$ of $F_m$ is defined on $U_{m-1}$ for $t\in[0,t_m]$, and $U_m=\phi_{t_m}^m(U_{m-1})$ is a neighbourhood of $q$ in $V_m$.
\item[(g)]  $\phi_{t_m}^m\circ\dots\circ\phi_{t_1}^1(p)=q$.
\end{enumerate}
For $j=1,\dots,m$ and $t\in[t_1+\dots+t_{j-1}, t_1+\dots+t_j]$, we have a biholomorphism 
$$\psi_t=\phi_{t-t_1-\dots-t_{j-1}}^j\circ\phi_{t_{j-1}}^{j-1}\circ\dots\circ\phi_{t_1}^1$$ 
defined on $U$.  Let $\tau=t_1+\dots+t_m$.  For each $a\in U$, the path $t\mapsto\psi_t(a)$, $t\in[0,\tau]$, in the leaf of $a$ is continuous and, by (c) above, injective.

Now $h=\psi_\tau:U\to U_m$ is a biholomorphism with $h(p)=q$ such that $h(x)$ lies in the leaf of $x$ for all $x\in U$, so condition (2) is verified.  As for condition (3), say $a\in U$ and $b$ is in the leaf of $a$ between $a$ and $h(a)$.  There is $s\in[0,\tau]$ such that $b=\psi_s(a)$.  Then $k=\psi_s$ is a biholomorphism from $U$ onto an open neighbourhood of $b$, taking $a$ to $b$, such that $k(x)$ lies in the leaf of $x$ between $x$ and $h(x)$ for every $x\in U$.
\end{proof}

We end this section by stating a global schlichtness lemma for vector fields.  Let $X$ be a complex manifold, $A$ be a closed subset of $X$, and $F$ be a smooth vector field on $X\setminus A$ with no compact orbits, so in particular, $F$ has no critical points.  Let us call $F$ {\it quasiholomorphic} if the foliation induced by $F$ on $X\setminus A$ is quasiholomorphic.  By Theorem~\ref{holomorphicatlas}, $F$ is quasiholomorphic if locally, $F$ is the product of a positive smooth function and a holomorphic vector field.  We call a domain $Y$ in $X$ an {\it interval domain with respect to} $F$, or simply an {\it interval domain} if $F$ is understood, if $Y$ contains $A$ and has a nonempty connected intersection with each orbit of $F$.

If $F$ is a quasiholomorphic vector field on an open subset of a complex manifold $X$ and $Y$ is an interval domain in $X$ with respect to $F$, it follows immediately from Theorem \ref{Trepreau1} that there is a largest interval domain $Z$ in $X$ with respect to $F$ containing $Y$ to which all holomorphic functions on $Y$ extend holomorphically, and $Z$ is locally Stein.

\section{Corollaries, examples, and counterexamples}

\noindent
Let $A$ be a nonsingular complex $n\times n$ matrix.  Applying Theorem \ref{Trepreau1} to the holomorphic vector field $F:z\mapsto Az$ on $\mathbb C^n$ yields the following generalization of the well-known fact that the envelope of holomorphy of a star-shaped domain in $\C^n$ is a star-shaped domain in $\C^n$.

\begin{corollary}  Let $A\in {\rm GL}(n,\C)$ and let $\Om$ be a domain in $\C^n$ such that the open set $\{t\in\R:e^{tA}z\in\Om\}$ is a nonempty interval for every $z\in\C^n$.  Then the envelope of holomorphy of $\Om$ is a Stein domain in $\C^n$ of the same kind. 
\label{linear}
\end{corollary}

The assumption that $\{t\in\R:e^{tA}z\in\Om\}$ is a nonempty interval for every $z\in\C^n$ implies that $\Om$ contains all the compact orbits of $F$.  It is a little exercise in linear algebra to show that the union of these orbits is closed.

\smallskip

Next we look at a variant of Theorem \ref{Trepreau1} for the special case of a backwards complete holomorphic vector field.  In this case we are able to describe envelopes with respect to certain subfamilies of the set of all holomorphic functions.  Let $F$ be a smooth vector field on a smooth manifold $X$.  A domain $\Om$ in $X$ is called a {\it half-space with respect to} $F$ if $\Om$ intersects every orbit of $F$ in a backwards semiorbit.  It follows that $\Om$ contains the cycles and equilibria of $F$.  We say that $F$ is {\it backwards complete} if its flow is defined on all of $X$ for all negative time.  In other words, the maximal integral curves of $F$ are defined on intervals that are unbounded below.  A straightforward modification of the proof of Theorem \ref{Trepreau1} gives the following result.

\begin{theorem}  Let $F$ be a backwards complete holomorphic vector field on a complex manifold $X$ and let $Y$ be a half-space in $X$ with respect to $F$.  Let $\mathcal H\subset\O(Y)$ be nonempty and closed under precomposition by the backwards time-maps of $F$.  

There is a largest half-space $Z$ in $X$ with respect to $F$ containing $Y$ to which all holomorphic functions in $\mathcal H$ extend holomorphically.  Let $\widetilde{\mathcal H}\subset\O(Z)$ be the set of the extensions to $Z$ of the functions in $\mathcal H$.  There is no domain $U$ in $X$ with $U\not\subset Z$ and a domain $V$ in $Z\cap U\neq\varnothing$ such that every function in $\widetilde{\mathcal H}$ agrees on $V$ with a holomorphic function on $U$.  

Hence, $Y$ has a schlicht $\mathcal H$-envelope in $X$, and the $\mathcal H$-envelope is a half-space with respect to $F$.
\label{backwards}
\end{theorem}

The hypothesis that $\mathcal H$ be closed under precomposition by the backwards time-maps of $F$ cannot be omitted.  As a simple example, let $X=\C$, $F=-\partial/\partial x$, $Y=\C\setminus(-\infty,0]$, and $\mathcal H=\{f\}$, where $f$ is the branch of the square root on $Y$ that maps the positive real axis to itself.  Then $Y$ itself is the largest half-space in $X$ with respect to $F$ containing $Y$ to which $f$ extends holomorphically, but $Y$ is not its own $\mathcal H$-envelope: this is the domain $\C\setminus\{0\}\to\C$, $z\mapsto z^2$, over $\C$.  On the other hand, $Y$ is its own envelope with respect to the set $\{f(\cdot+t):t\geq 0\}$, which is the smallest subset of $\O(Y)$ that contains $f$ and is closed under precomposition by the backwards time-maps of $F$. 

\smallskip
\noindent
{\bf A counterexample to a more general schlichtness lemma.}  We now present an example showing that our global schlichtness lemma does not hold for real-analytic vector fields, or even for polynomial vector fields on $\C^n$, if the assumption of quasiholomorphicity is omitted. 

Consider the vector field $F_0=\dfrac\partial{\partial x_1}+(3x_1^2-1)\dfrac\partial{\partial x_2}$ on $\R^2$.  The integral curve $\gamma$ of $F_0$ with $\gamma(0)=(0,c)$, $c\in\R$, is $\gamma(t)=(t,t^3-t+c)$.  Let $\alpha=1/\sqrt{3}$ and 
$$Y_0=\{(x_1,x_2)\in\R^2:x_2<\frac 2 3 \alpha \text{ and if }-\alpha<x_1<2\alpha, \text{ then }x_2<x_1^3-x_1\}.$$
Then $Y_0$ is an interval domain with respect to $F_0$, but its convex hull $Z_0$ is not.

\begin{figure}[h]
{\resizebox{4cm}{5cm}{\includegraphics{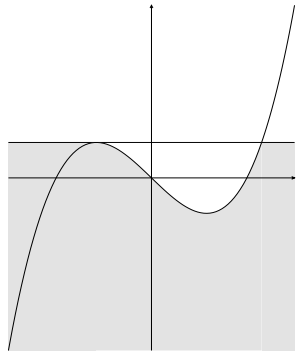}}}
\caption{The shaded domain in the plane $\R^2$ is $Y_0$.}
\end{figure}

Let $F$ be the extension of $F_0$ to $\C^2=\R^2+i\R^2$ that is independent of the imaginary coordinates.  Then the tube domain $Y=Y_0+i\R^2$ is an interval domain with respect to $F$, but its envelope of holomorphy $Z_0+i\R^2$ is not.  In fact, the largest interval domain with respect to $F$ containing $Y$ to which all holomorphic functions on $Y$ extend is just $Y$ itself.  It follows that $F$ is not quasiholomorphic.  It also follows that $Y$ is not an interval domain with respect to the holomorphic extension
$$\dfrac\partial{\partial x_1}+(3x_1^2-3y_1^2-1)\dfrac\partial{\partial x_2}+6x_1y_1\dfrac\partial{\partial y_2}$$
of $F_0$ to $\C^2$.

\smallskip

The following two propositions shed light on the notion of quasiholomorphicity in the 1-dimensional case.  An antiholomorphic vector field on a domain $\Om$ in $\C$ is a vector field of the form $F=\bar g \dfrac\partial{\partial z}$, where $g$ is a holomorphic function on $\Om$.  Such a field is not holomorphic, of course, unless $g$ is constant.  It is easily seen that the antiholomorphic vector fields are precisely the fields that are locally of the form $F=\nabla u$, where $u$ is harmonic.  Then $u$ is strictly increasing along the nontrivial orbits of $F$.  If $\Om$ is simply connected, then $F$ has a potential $u$ defined on all of $\Om$, so $F$ has no cycles.

\begin{proposition}  An antiholomorphic vector field on a domain $\Om$ in $\C$ is quasiholomorphic if and only if all its orbits are noncompact.  This holds for example if $\Om$ is simply connected.
\label{antiholomorphic}
\end{proposition}

Note that all the orbits of the antiholomorphic field $\nabla\arg=\dfrac 1{x^2+y^2}(-y,x)$ on $\C\setminus\{0\}$ are compact.

\begin{proof}  The vector field $\bar g \dfrac\partial{\partial z}$, where $g$ is a holomorphic function on $\Om$ with no zeros, is the product of the positive smooth function $|g|^2$ and the holomorphic vector field $\dfrac 1 g \dfrac\partial{\partial z}$.
\end{proof}

The next result shows that quasiholomorphicity of vector fields is not preserved by addition.

\begin{proposition}  Let $a,b\neq 0$ be real numbers.  The vector field $F=ax\dfrac\partial{\partial x}+by\dfrac\partial{\partial y}$, that is, the real-linear vector field given by the diagonal matrix $\left[\begin{array}{cc} a & 0 \\ 0 & b \end{array} \right]$, is quasiholomorphic on $\C\setminus\{0\}$ if and only if $F$ is holomorphic or antiholomorphic, that is, $a=\pm b$.
\label{reallinear}
\end{proposition}

\begin{proof}  We start by proving a general fact for an arbitrary smooth vector field $F$ on a domain $\Om$ in $\C$.  Suppose there is a holomorphic map $h:U\to V$ between open subsets of $\Om$ such that $h(x)$ lies in the $F$-orbit of $x$ for all $x\in U$.  Let $\alpha:U\cup V\to\R$ be a smooth function with no critical points which is constant on the intersection of each $F$-orbit with $U\cup V$.  Then $\alpha\circ h=\alpha$ on $U$.  Differentiating this equation with respect to $z$ gives $(\alpha_z\circ h)h'=\alpha_z$, so 
$$0=(h')_{\bar z}=\bigg(\dfrac{\alpha_z}{\alpha_z\circ h}\bigg)_{\bar z} = \frac{\alpha_{z\bar z}(\alpha_z\circ h)-\alpha_z(\alpha_{z\bar z}\circ h)\overline{h'}}{(\alpha_z\circ h)^2}.$$
Substituting $\overline{h'}=\dfrac{\alpha_{\bar z}}{\alpha_{\bar z}\circ h}$ yields
$$\frac{\alpha_{z\bar z}}{|\alpha_z|^2}\circ h = \frac{\alpha_{z\bar z}}{|\alpha_z|^2}$$
on $U$.  In summary, the existence of a local holomorphic map $h$ taking each $F$-orbit into itself implies that if $\alpha$ is a conserved quantity for $F$, then so is $\Delta\alpha/|\nabla\alpha|^2$.

All orbits of $F=ax\dfrac\partial{\partial x}+by\dfrac\partial{\partial y}$ (except the equilibrium at $0$) are noncompact: they are the images of the integral curves $t\mapsto (e^{at}x, e^{bt}y)$ for $(x,y)\in\R^2\setminus\{0\}$.  We can take $\alpha(x,y)=-b\log x+a\log y$ as a conserved quantity for $F$, in the first quadrant, say.  Then
$$\frac{\Delta\alpha}{|\nabla\alpha|^2} = \frac{by^2-ax^2}{b^2y^2+a^2x^2}.$$
We leave it as an exercise for the reader to show that this quantity is not constant on the orbits of $F$ unless $a=\pm b$.  This shows that $F$ is not quasiholomorphic unless $a=\pm b$.
\end{proof}

We are grateful to the referee for suggesting how to continue the above calculation so as to prove the converse of Theorem \ref{holomorphicatlas} in a special case.

\begin{proposition}  Let $\mathcal F$ be a smooth 1-dimensional foliation of a Riemann surface $X$ satisfying the following two conditions.
\begin{enumerate}
\item[(a)]  If $p,q\in X$ lie in the same leaf, then there is an open neighbourhood $U$ of $p$ in $X$ and a biholomorphism $h$ from $U$ onto an open neighbourhood of $q$, taking $p$ to $q$, such that $h(x)$ lies in the leaf of $x$ for all $x\in U$.  (This is part {\rm (2)} of our definition of quasiholomorphicity.)
\item[(b)]  There is a cover of $X$ by open subsets $U$ with submersions $\alpha:U\to\R$, such that the nonempty intersections with $U$ of the leaves of $\mathcal F$ are precisely the fibres of $\alpha$.
\end{enumerate}
Then $\mathcal F$ has a holomorphic atlas.
\end{proposition}

\begin{proof}  Let $U$ be an open subset of $X$ with a surjective submersion $\alpha:U\to\R$ as in (b).  We may assume that $U$ is a coordinate neighbourhood with a coordinate $z$.  Using (a) as in the proof of Proposition \ref{reallinear}, we can show that the function $\alpha_{z\bar z}/|\alpha_z|^2$ on $U$ is constant on the leaves of $\mathcal F$, that is, on the fibres of $\alpha$.  Hence, there is a smooth function $u:\R\to\R$ with $\alpha_{z\bar z}/|\alpha_z|^2=u\circ\alpha$.  Let $v$ be an antiderivative for $u$, and $w$ be an antiderivative for $e^{-v}$.  Then $w''=-uw'$, so
$$(w\circ\alpha)_{z\bar z}=(w'\circ\alpha)\alpha_{z\bar z}+(w''\circ\alpha)|\alpha_z|^2 = |\alpha_z|^2(w''\circ\alpha+(u\circ\alpha)(w'\circ\alpha))=0,$$
and $w\circ\alpha$ is harmonic.  Also, $w'=e^{-v}$ is nowhere zero, so $w$ is injective and $(w\circ\alpha)_z=(w'\circ\alpha)\alpha_z$ is nowhere zero.  Thus, $w\circ\alpha$ is a harmonic submersion whose nonempty fibres are the same as those of $\alpha$.  Locally, $w\circ\alpha$ is the imaginary part of a holomorphic function $f$, and the nowhere-vanishing holomorphic vector field $\partial/\partial\,\text{Re}f$ induces $\mathcal F$.
\end{proof}

We have been unable to generalize this result to higher dimensions.

\section{Local schlichtness at a nondegenerate boundary singularity}

\noindent
By the schlichtness lemma, a domain is locally schlicht at a smooth boundary point.  We are now able to prove local schlichtness at a well-behaved isolated boundary singularity.

A real quadratic form $\sigma$ on $\C^n$ is of the form
$$\sigma(z)=\bar z^t H z+\mathrm{Re}(z^t Sz),$$
where $S$ and $H$ are complex $n\times n$ matrices, $S$ is symmetric and $H$ is hermitian.  We call $\bar z^t H z$ the hermitian part of $\sigma$, and $\mathrm{Re}(z^t Sz)$ the harmonic part of $\sigma$.  We call $\sigma$ hermitian if its harmonic part vanishes, and harmonic if its hermitian part vanishes.  The gradient $\nabla\sigma$ of $\sigma$ is a real-linear vector field on $\C^n$.  The complex-linear part of $\nabla\sigma$ is the gradient of the hermitian part of $\sigma$, and the conjugate-linear part of $\nabla\sigma$ is the gradient of the harmonic part of $\sigma$.

Let $\rho$ be a smooth real-valued function on a neighbourhood of the origin $0$ in $\C^n$ with a nondegenerate critical point at $0$.  Then the Taylor series of $\rho$ at $0$ is $\rho(0)+\sigma+\text{higher-order terms}$, where $\sigma$ is a nondegenerate real quadratic form.  We say that the critical point $0$ of $\rho$ is hermitian if $\sigma$ is hermitian.  

\begin{theorem}  Let $\rho$ be a smooth real-valued function on a neighbourhood of the origin $0$ in $\C^n$.  Suppose $\rho$ has a nondegenerate hermitian critical point at $0$, which is not a minimum.  Then the open set $\{\rho<\rho(0)\}$ is locally schlicht at $0$.
\label{boundary}
\end{theorem}

Before proceeding to the proof, we observe that this result has no overlap with Corollary \ref{subgraph}.  By Morse's lemma (see \cite{Milnor}, Lemma 2.2), there are smooth coordinates $x_1,\dots,x_{2n}$ at $0$ such that
$$\rho=\rho(0)-x_1^2-\dots-x_j^2+x_{j+1}^2+\dots+x_{2n}^2$$
near $0$.  Here, $j$ is the index of $\rho$ at $0$, and $1\leq j\leq 2n$.  It is then easily seen that the local homotopy type of $\{\rho<\rho(0)\}$ at $0$ is the sphere of dimension $j-1$.  On the other hand, the domain $\Om$ in Corollary \ref{subgraph} is locally contractible at the boundary point in question.

\begin{proof}  The result is trivially true if $\rho$ has a maximum at $0$, so we assume that this is not the case and that $n\geq 2$.  We may also assume that $\rho(0)=0$.  Let $\sigma$ be the quadratic part of $\rho$ at $0$.  Write $\sigma(z)=\bar z^t Hz$, where $H$ is an $n\times n$ nonsingular hermitian matrix.  By precomposing $\rho$ by a suitable complex-linear map, we can make $H$ diagonal with positive eigenvalues $a_1,\dots,a_k$ and negative eigenvalues $a_{k+1},\dots,a_n$, where $1\leq k\leq n-1$, so $\sigma(z)=\sum\limits_{j=1}^n a_j|z_j|^2$.  The gradient of $\sigma$ is the holomorphic vector field $\nabla\sigma=2\sum\limits_{j=1}^n a_j z_j \dfrac\partial{\partial z_j}$ with integral curves $\gamma(t)=(z_1 e^{a_1 t},\dots, z_n e^{a_n t})$.  We have 
$$\dfrac d{dt}\rho(\gamma(t)) = \nabla\rho(\gamma(t))\cdot\gamma'(t)=\nabla\rho(\gamma(t))\cdot\nabla\sigma(\gamma(t)).$$
Now $\nabla\rho\cdot\nabla\sigma$ is the sum of $\|\nabla\sigma\|^2$ and terms of higher order, and is therefore positive on a punctured neighbourhood of $0$, so $\rho$ is strictly increasing along the nontrivial orbits of $\nabla\sigma$ near $0$.  Let $E=\C^k\times\{0\}$ be the subspace of $\C^n$ corresponding to the positive eigenvalues.  It is saturated with respect to the flow of $\nabla\sigma$.  Also, $\rho\geq 0$ on $E$ near $0$.  

Let $\epsilon=\frac 1 2\min\{a_1,\dots,a_k\}>0$ and let $c_j=\dfrac{a_j}{a_j-\epsilon}>0$ for $j=1,\dots,n$.  For $r>0$, let $U_r$ be the open ellipsoid centred at $0$, convex and hence Stein, defined by the inequality $\lambda(z)=\sum\limits_{j=1}^n c_j|z_j|^2 < r^2$.  Let $Y_r=\{z\in U_r:\rho(z)<0\}$.  We claim that for $r>0$ small enough, $Y_r$ has a schlicht envelope of holomorphy in $U_r$.  This proves the theorem.  

To prove our claim, it suffices to show that for $r>0$ small enough, $Y_r$ is an interval domain in $U_r\setminus E$ with respect to $\nabla\sigma$.  Namely, by Theorem \ref{Trepreau1}, there is then a largest interval domain $Z_r$ in $U_r\setminus E$ with respect to $\nabla\sigma$ containing $Y_r$ to which all holomorphic functions on $Y_r$ extend holomorphically, and $Z_r$ is locally Stein in $U_r\setminus E$.  If $k=\dim E=n-1$, then $U_r\setminus E$ is Stein, so $Z_r$ is Stein.  Then $Z_r$ is the envelope of holomorphy of $Y_r$, and the claim is proved.

If $E$ has codimension $n-k\geq 2$, we complete $Z_r$ to a domain $W_r$ in $U_r$ by adding to $Z_r$ all those points in $E$ that have a neighbourhood $V$ with $V\setminus E\subset Z_r$.  In other words, $W_r$ is the largest domain in $U_r$ containing $Z_r$.  Every holomorphic function on $Z_r$ extends to $W_r$.  We know that $W_r$ is locally Stein at its boundary points outside $E$.  In the sense of Grauert and Remmert \cite{GrauertRemmert}, the accessible boundary points of $W_r$ over $E$ are nonremovable and form a thin closed set in the accessible boundary of $W_r$.  Hence, by Satz 4 in \cite{GrauertRemmert}, $W_r$ is Stein, so $W_r$ is the envelope of holomorphy of $Y_r$, and the claim is proved.

It remains to show that for $r>0$ small enough, $Y_r$ is an interval domain in $U_r\setminus E$ with respect to $\nabla\sigma$.  First note that each orbit of $\nabla\sigma$ intersects $U_r$ in an interval (possibly empty) because 
$$(\lambda\circ\gamma)''(t)=\frac {d^2}{dt^2}\sum\limits_{j=1}^n c_j|z_j|^2e^{2a_jt} = 4\sum\limits_{j=1}^n a_j^2 c_j|z_j|^2e^{2a_jt}$$
is always nonnegative.  Since, as observed earlier, $\rho$ is strictly increasing along the nontrivial orbits of $\nabla\sigma$ near $0$, it follows that for $r>0$ small enough, each orbit of $\nabla\sigma$ intersects $Y_r$ in an interval (possibly empty).

Finally, we must show that for $r>0$ small enough, each orbit of $\nabla\sigma$ that intersects $U_r\setminus E$ also intersects $Y_r$.  An orbit that intersects $U_r\setminus E$ enters $U_r$ at a point $z$ where $\nabla\sigma$ points into $U_r$, that is,
$$\nabla\sigma\cdot\nabla\lambda = 4\sum\limits_{j=1}^n a_j c_j|z_j|^2 < 0,$$
so
$$\sigma(z)+\epsilon r^2 = \sigma(z)+\epsilon\lambda(z) = \sum\limits_{j=1}^n (a_j+\epsilon c_j)|z_j|^2 = \sum\limits_{j=1}^n a_j c_j|z_j|^2 < 0.$$
Thus, for $r>0$ small enough,
$$\rho(z)=\sigma(z)+O(\|z\|^3)< -\epsilon r^2+O(r^3) < 0,$$
so the orbit intersects $Y_r$.
\end{proof}

The referee has suggested an alternative proof in case $k\leq n-2$.  First, in the model case $\rho=\sigma$, the Kugelsatz applied with fixed $z_1,\dots,z_k$ provides holomorphic extensions to neighbourhoods of $0$.  Second, it may be argued that this extension property is stable under perturbation of $\sigma$ by higher-order terms.

We do not know whether the assumption that the singularity is hermitian can be omitted, but the following result suggests that our methods cannot take Theorem \ref{boundary} beyond the hermitian case.

\begin{theorem}  Let $\sigma$ be a nondegenerate real quadratic form on $\C^n$.  There is a positive smooth function $u$ on a domain $\Om$ in $\C^n$ such that $u\nabla\sigma$ is holomorphic on $\Om$ (so $\nabla\sigma$ is quasiholomorphic on $\Om$) if and only if
\begin{enumerate}
\item  $\sigma$ is hermitian, that is, $\nabla\sigma$ is complex-linear, or
\item  $n=1$ and $\sigma$ is harmonic, so $\sigma(z)=\mathrm{Re}(az^2)$ for some $a\in\C$, $a\neq 0$.
\end{enumerate}
\label{disappointment}
\end{theorem}

\begin{proof}  Write $\sigma(z)=\bar z^t H z+\mathrm{Re}(z^t Sz)$, where $S$ and $H$ are complex $n\times n$ matrices, $S$ is symmetric and $H$ is hermitian.  Suppose there is a positive smooth function $u$ on a domain $\Om$ in $\C^n$ such that $u\nabla\sigma$ is holomorphic on $\Om$.  Denote the standard coordinates on $\C^n$ as $z_k=x_k+iy_k$, $k=1,\dots,n$.  Holomorphicity of $u\nabla\sigma$ means that
$$u\bigg(\frac{\partial\sigma}{\partial x_k}+i\frac{\partial\sigma}{\partial y_k}\bigg)=u\frac{\partial\sigma}{\partial\bar z_k}$$ 
is holomorphic, that is, 
$$0=\frac\partial{\partial z_j}\bigg(u\frac{\partial\sigma}{\partial z_k}\bigg) = \frac{\partial u}{\partial z_j}\frac{\partial\sigma}{\partial z_k}+u\frac{\partial^2\sigma}{\partial z_j\partial z_k}$$
for $j,k=1,\dots,n$.  Thus, the Hessian matrix $\bigg[\dfrac{\partial^2\sigma}{\partial z_j\partial z_k}\bigg]$, which is simply $S$, has rank $0$ or $1$.  If $S=0$, then (1) holds.  Suppose $S$ has rank $1$.  We need to show that $n=1$ and $H=0$.

After a complex-linear change of coordinates, $S_{jk}=0$ except for $j=k=1$.  Hence,
$$0=\frac{\partial\sigma}{\partial z_k} = \sum_{j=1}^n H_{jk}\bar z_j$$
for $k\geq 2$, so $H_{jk}=0$ for $j\geq 1$ and $k\geq 2$.  Since $H$ is hermitian, it follows that $H_{jk}=0$ except for $j=k=1$.  This implies that $\sigma$ is degenerate if $n\geq 2$, so we conclude that $n=1$.

Write $\sigma(z)=a|z|^2+\mathrm{Re}(bz^2)$ with $a\in\R$ and $b\in\C$.  We have $u_z\sigma_z=-u\sigma_{zz}$.  Since $u$ is real,
$$(\log u)_{z\bar z} = \bigg(\frac{u_z}u\bigg)_{\bar z} = -\bigg(\frac{\sigma_{zz}}{\sigma_z}\bigg)_{\bar z} = -\bigg(\frac b{a\bar z+bz}\bigg)_{\bar z} = \frac{ab}{(a\bar z+bz)^2}$$
is real for all $z\in\Om$.  Thus, $a=0$ or $b=0$, which is what we want to prove, or $(a\bar z+bz)^2/b$ is real, that is, $\sqrt{b} z+a\bar z/\sqrt{b}$ is real or imaginary for all $z\in\Om$.  By a brief calculation, it follows that $a=\pm|b|$, which means precisely that $\sigma$ is degenerate.
\end{proof}

\end{document}